\documentclass[11pt]{article}
\usepackage{epigamath}

\setpapertype{A4}

\usepackage[english]{babel}

\usepackage[dvips]{graphicx}     

\title{Conic bundles that are not birational to \\ numerical Calabi--Yau pairs}
\titlemark{Numerical Calabi--Yau pairs}
\author{J\'anos Koll\'ar}
\authoraddresses{
\authordata{J. Koll\'ar}{\firstname{J\'anos} \lastname{Koll\'ar}\\
\institution{Princeton University, Princeton NJ 08544-1000}\\
\email{kollar@math.princeton.edu}}}

\authormark{J.\ Koll\'ar}
\date{\vspace{-5ex}} 
\journal{\'Epijournal de G\'eom\'etrie Alg\'ebrique} 
\acceptation{Received by the Editors on June 24, 2016, and in final form on January 17, 2017. \\ Accepted on January 30, 2017.}

\acknowledgement{Partial financial support was provided  by  the NSF under grant number DMS-1362960.}

\newcommand{\articlereference}{Volume 1 (2017), Article Nr. 1}

\let \cedilla =\c
\renewcommand{\c}[0]{{\mathbb C}}  

\renewcommand{\o}[0]{{\mathcal O}} 
\newcommand{\z}[0]{{\mathbb Z}}

\renewcommand{\r}[0]{{\mathbb R}} 

\renewcommand{\a}[0]{{\mathbb A}}

\newcommand{\p}[0]{{\mathbb P}}
\newcommand{\f}[0]{{\mathbb F}}
\newcommand{\q}[0]{{\mathbb Q}}
\newcommand{\map}[0]{\dasharrow}
\newcommand{\qtq}[1]{\quad\mbox{#1}\quad}

\newcommand{\pic}[0]{\operatorname{Pic}}

\newcommand{\supp}[0]{\operatorname{Supp}}    
\newcommand{\red}[0]{\operatorname{red}}

\newcommand{\bs}[0]{\operatorname{Bs}}

\newcommand{\sing}[0]{\operatorname{Sing}}

\newcommand{\chr}[0]{\operatorname{char}}

\newcommand{\simq}[0]{\sim_{\q}}

\newcommand{\tsum}[0]{\textstyle{\sum}}

\def\into{\DOTSB\lhook\joinrel\to}

\def\loccoh#1.#2.#3.#4.{H^{#1}_{#2}(#3,#4)}

\DeclareMathAlphabet{\mathchanc}{OT1}{pzc}%
                                {m}{it}

\newcommand{\simb}[0]{\stackrel{bir}{\sim}}

\newcommand{\sym}[0]{\operatorname{Sym}}

\newtheorem{thm}{Theorem}

\newtheorem{lem}[thm]{Lemma}
\newtheorem{cor}[thm]{Corollary}

\newtheorem{prop}[thm]{Proposition}

\newtheorem{complem}[thm]{Complement}
   
\newtheorem{defn}[thm]{Definition}

\newtheorem{say}[thm]{\!\!}
\newtheorem{exmp}[thm]{Example}

\newtheorem{rem}[thm]{Remark}

\newtheorem{defn-thm}[thm]{Definition--Theorem}  
\newtheorem{defn-lem}[thm]{Definition--Lemma}  

\newtheorem{main-exmp}[thm]{Main Example}
\newtheorem{baby-exmp}[thm]{Baby Example}

\numberwithin{equation}{thm}

\newcommand{\problem}{\refstepcounter{equation} {\it Problem \arabic{thm}.\arabic{equation}.}}
\newcommand{\Claim}{\refstepcounter{equation} {\it Claim \arabic{thm}.\arabic{equation}.}}

\begin{document}


\maketitle

\begin{prelims}

\def\abstractname{Abstract}
\abstract{Let $X$ be a general conic bundle over ${\mathbb P}^2$
with branch curve of degree at least 19. We prove that there is no normal projective variety $Y$ that is birational to $X$ and such that some multiple of its anticanonical divisor is effective.}



\vspace{0.2cm}

\languagesection{Fran\cedilla{c}ais}{%

\vspace{0.05cm}
\textbf{Titre. Fibr\'es en coniques qui ne sont pas birationnels \`a des paires de Calabi--Yau num\'eriques}
\commentskip
\textbf{R\'esum\'e.} Soit $X$ un fibr\'e en coniques g\'en\'eral sur ${\mathbb P}^2$
avec une courbe de branchement de degr\'e au moins 19. Nous montrons qu'il n'existe pas de vari\'et\'e projective normale $Y$ qui soit birationnelle \`a $X$ 
et telle qu'un multiple de son diviseur anticanonique soit effectif.}


\end{prelims}


\newpage
  Understanding the difference between rationally connected varieties and Fano varieties has long been a goal of birational geometry.
In any dimension, smooth Fano varieties come in finitely many families  
but rationally connected varieties form infinitely many families.
Thus one expects that not every rationally connected variety is birational to a Fano variety, but actual proofs of this fact have been quite subtle; see 
\cite{Sar81} or \cite{MR1311348}.

A more general form of this problem asks if every rationally connected variety $X$ is birational to the underlying variety of a  $\q$-Fano pair $(Y, \Delta)$.
Since  $\q$-Fano pairs form infinitely many families up-to birational equivalence \cite{MR2520058}, this is a harder variant.  This form of the question  was  posed in \cite{cas-gon}
and a negative answer is established  in \cite{kry}.
Closely related results describing 
Calabi--Yau fiber-space structures on certain Fano 3-folds are 
proved in \cite{chel-2004, MR2223680, MR2565543}.

Probably the most general question in this direction is whether
every rationally connected variety is birational to the underlying variety of a  {\it numerical Calabi--Yau} pair $(Y, \Delta)$. Here we allow the most general
definition of Calabi--Yau pairs, that is, $Y$ is normal, proper, 
$\Delta$ is a pseudo-effective $\r$-divisor, 
$K_Y+\Delta$ is $\r$-Cartier  and $K_Y+\Delta\equiv 0$, but we impose no restrictions on the singularities of $Y$. 
(For most purposes the more restrictive definition of 
Calabi--Yau pairs adopted in \cite{k-x-cy} is the best; but the above numerical version also seems natural.) 
It turns out that  the singularities of $Y$ are not very important 
and the difference between  effective and pseudo-effective divisors also
may not be significant,
but
allowing divisors with coefficient $>1$ in $\Delta$ leads to many more cases; see Example~\ref{product.exmp} and Definition~\ref{pot.eff.defn}. 
If $X$ is birational to a  numerical Calabi--Yau pair $(Y, \Delta)$
then the birational transform of $\Delta$ on $X$ is frequently a
quite interesting divisor. Understanding such divisors was a key  to proving  unirationality of degree 1 conic bundle surfaces  \cite{k-mel}.

We discuss several methods to show that certain conic bundles 
(see Definition~\ref{cb.defn.1})    are not
birational to any  numerical Calabi--Yau pair. 
Typical results are the following.

\begin{thm} \label{main.surf.thm} There are   conic bundles
$S\to \p^1$  defined over $\q$   that are
not birational to any numerical Calabi--Yau surface.
\end{thm}

This property seems to depend very subtly on the coefficients involved in the definition of $S$ and we give only  sufficient conditions in  Example \ref{over.Q.exmps.exmp}. The following are some  concrete  special cases.

\begin{exmp} 
\rm Let $p\geq 11$ be a prime such that $p\equiv -1\mod 4$ and choose $m\in \z$ not divisible by $p$.
Then the conic bundle 
$$
\Bigl(z^2=\Bigl(\tfrac{s^{p-1}-1}{s^2-1}\Bigr) x^2+\bigl(s^{p-1}-1+mp\bigr) y^2\Bigr)
\subset \p^2_{xyz}\times \a^1_s
$$
is not birational---over $\q$---to any numerical Calabi--Yau surface.
\end{exmp}

In dimension 3 we get the following.

\begin{thm} \label{main.thm} 
Let $X_{d,2}\subset \p^2\times  \p^2$ be  a general hypersurface of bidegree $(d,2)$ 
over a field of characteristic $\neq 2,3,5$.
 Then $X_{d,2}$ is not birational to the underlying variety of a numerical Calabi--Yau pair for   $d\geq 7$.
\end{thm}

I could not write down explicit examples with smooth branch curve,
but it is easy to get many with reducible branch curve  using Corollary 
\ref{good.nodes.lem.cor}.

It is quite likely that the theorem also holds over any infinite field. The main arguments in our paper work whenever the  characteristic is $\neq 2$, but a key reference  \cite{Sar81}
is stated in the literature only for  characteristic $\neq 2,3$ and
we also use resolution of singularities.

The most important numerical invariant of a  conic bundle
$\pi:X\to \p^n$ is the degree of its branch divisor $B_X\subset\p^n$.
Typical results say that for smooth branch divisors the degree is the only important invariant and  
the higher the degree of $B_X$, the more complicated $X$ is.
This is, however, not the case for our question. We see 
in Examples \ref{K3.exmp}--\ref{22.exmp} that 
there are conic bundles
$\pi:X\to \p^2$ with smooth branch curve of arbitrary high degree
(resp.\ $\pi:S\to \p^1$ with many singular fibers) 
 that are  birational to  numerical Calabi--Yau varieties.

The following example shows that many varieties are   birational to the underlying variety of a smooth  numerical  Calabi--Yau pair.

\begin{exmp} \label{product.exmp}
\rm If $(Y, \Delta)$ is  a  numerical Calabi--Yau pair and $\Delta\neq 0$
then $Y$ is uniruled by \cite{mi-mo} but $Y$ need not be rationally connected.
In fact,   for   any smooth, projective variety $X$,
the product $X\times \p^1 $ is  birational to a  numerical Calabi--Yau pair.
To see this, let  $H$ be an ample divisor on $X$ such that $H-K_X$ is effective and note that    the anticanonical class of $\p_X\bigl(\o_X\oplus\o_X(H)\bigr)$
is effective.  It is the sum of twice the negative section and of
the pull-back of  $H-K_X$.

In particular,  the conic bundles in Theorem \ref{main.thm} are not birationally ruled. However, we use  even  stronger  non-rationality results of \cite{Sar81}  during the proof. 
\end{exmp}

\begin{defn} \label{cb.defn.1}
\rm A {\it conic bundle} 
is a flat, proper morphism $\pi:X\to Z$ such that every fiber of $\pi$ is isomorphic to a plane conic.  
We are mainly interested in the cases when $Z$ is regular, but for the basic definition it is enough to assume that $Z$ is normal.
(For many purposes one should allow non-flat morphisms and more singular fibers, but for us the  restrictive version is more convenient.) 

A  conic bundle is called {\it extremal}
if for every codimension 1 point $z\in Z$ the fiber $X_z$ is irreducible over $k(z)$. Equivalently, the relative Picard number is 1.  

A  conic bundle is called {\it minimal} if it is extremal and
has no rational sections. (This is the ``right'' definition for conic bundles but note that 
 a product  $Z\times \p^1\to Z$ is extremal but not  minimal in our sense.)
If $X$ is regular then  $\pi:X\to Z$ is minimal
iff $\pic(X)=\pi^*\pic(Z)\oplus\z[\omega_{X/Z}^{-1}]$; this is the key property that we are interested in.
This is equivalent to saying that if $L$ is any line bundle on $X$ and
$C\subset X$ is  a fiber then  $\deg (L|_C)$ is even and if
$C$ is reducible then $L|_C$ has the same degree on both
irreducible components.

The {\it branch locus,} denoted by $B_X\subset Z$,  is the 
subscheme parametrizing singular fibers of $\pi$. Set theoretically
it consists of points $z\in Z$ such that
$X_z:=\pi^{-1}(z)$ is singular.  If $X$ is regular then $B_X$ is reduced.
In general, the scheme structure is given by the formula
(\ref{extend.say.3}) which also shows that 
$B_X$ has pure codimension 1 in $Z$,
except possibly when  $\chr k(Z)= 2$.
In order to avoid various complications, we
assume  from now on that all residue characteristics are $\neq 2$.
Our main interest is in conic bundles over $\c$ or $\q$, but we will
use some examples that are defined over $\z_p$.

Assume that $Z$ is  regular and let $z\in B_X$ be a regular point. Then 
 $X_z$ is a pair of lines and $X$ is regular along $X_z$; see Paragraph \ref{extend.say}.
This defines a
double cover  $\tilde B_X\to B_X$ that is \'etale  over the regular locus of $B_X$. Then $X\to Z$ is minimal
iff, for every irreducible component $B_i\subset B_X$, the corresponding
$\tilde B_i$ is also irreducible. 

I do not know a good general introduction to conic bundles, but everything can be gleaned from the basic sources \cite{MR0220734, MR0472843, Sar81, MR899398}  or \cite[Chap.4]{MR1798984}, \cite[3.11-13]{ksc}. 
The key results are discussed in Paragraphs \ref{extend.say}--\ref{sark.thm.c}.
\end{defn}

\begin{defn} \label{pot.eff.defn}
\rm Let $k$ be a field and $X$ a normal, proper  $k$-variety of dimension $n$.
A divisor class $B$ is called {\it effective} (resp.\ {\it $\q$-effective})
if $B$ is linearly  (resp.\  $\q$-linearly) equivalent to an
effective $\z$-divisor  (resp.\  $\q$-divisor). An $\r$-divisor
$B$ is called {\it pseudo-effective} if its class in
$\operatorname{N}_{n-1}(X)_{\r}$ 
is a limit of  $\q$-effective divisors. 

Birational transformation of   divisors does not preserve linear equivalence,
 so it is not very useful to ask whether a divisor is  birationally effective or not. 
However, the birational transform of a  mobile linear system is well defined
and the canonical class makes sense on any birational model of $X$.


Let  $|M|$ be a mobile linear system on $X$. We say that 
$-K_X+|M|$ is {\it birationally effective} 
(resp.\ {\it birationally $\q$-effective}  or {\it birationally pseudo-effective})
if there is a normal,
proper $k$-variety $\phi: X'\simb X$ such that $-K_{X'}+\phi_*|M|$ is  effective (resp.\ $\q$-effective or pseudo-effective).

For now our main interest is in the case $|M|=0$.

We stress that we do not impose  a priori restrictions on the singularities of $X'$  but it is easy to improve them, at least in characteristic 0.

Assume that $-K_{X'}\simq D'$ is pseudo-effective and 
let $p:X''\to X'$ be a terminal modification  \cite[1.33]{kk-singbook}.
Write $-K_{X''}\simq E''+D''$ where $E''$ is $p$-exceptional and $D''$ is the birational transform of $D'$. By the Negativity lemma  \cite[3.39]{km-book}
we see that $E''$ is an effective $\r$-divisor, hence $-K_{X''}$ is pseudo-effective.

If $K_{X''} $ is not numerically trivial, then it is also not
pseudo-effective, hence a suitable minimal model program terminates with a 
Mori fiber space  $X^m\to Z$. That is, $-K_{X^m}$ is relatively ample and
the relative Picard number of $X^m/Z$ is 1.
Note that  $\phi:X''\map X^m$ is a birational contraction, that is,
$\phi^{-1}$ has no exceptional divisors. Thus if $-K_{X''}\simq D''$
then $-K_{ X^m}=-\phi_*(K_{X''})\simq\phi_*(D'')=:D^m$ shows that $-K_{ X^m} $ is also pseudo-effective.

Thus 
 the following holds.
\medskip

\Claim \label{pot.eff.defn.1} 
Let $k$ be a field of characteristic 0 and $X$ a normal $k$-variety. Then $-K_X$ is  birationally pseudo-effective iff 
\begin{enumerate}
\item[(a)] either $X$ has canonical singularities and $K_X$ is numerically trivial,
\item[(b)]or 
there is a
projective variety $ X^m$  with terminal singularities such that
$ X^m$ is birational to $X$, $-K_{X^m} $ is  pseudo-effective
and there is
a Mori fiber space structure $\pi: X^m\to Z$.\qed
\end{enumerate}
\medskip

This  suggests that 
 our question should be treated using the
Noether--Fano method. (See \cite[Chap.5]{ksc} for an introduction and
\cite{MR1798984} for a more detailed treatment.)
This is  the approach taken in \cite{kry} and the examples given there are also not birational to a
numerical Calabi--Yau pair. 
We work with conic bundles. These have a rich birational geometry
yet all of their birational models are quite well understood.

The characteristic 0 assumption is only needed to guarantee that
resolutions and Mori fiber space models exist. Thus  Claim \ref{pot.eff.defn.1}
 holds over any field $k$ if $\dim X=2$ or if $\dim X=3$ and $\chr k>5$.

 Note also that in birational geometry one would usually like to control
the singularities of the pair  $(X',\Delta')$, not just the singularities of $X'$. However, even if $\Delta '$ is effective, we do not assume that
the coefficients in $\Delta'$ are $\leq 1$, thus  a similar reduction to the
log canonical case is not possible.
\end{defn}

\begin{rem} \rm Our notion is  close to the concept of {\it Fano type}  varieties introduced in \cite{MR2448282},
which  asks for a normal,
proper $k$-variety $X'\simb X$ such that $-K_{X'}\simq \Delta'+H'$
where $(X', \Delta')$ is klt and $H'$ is an ample $\q$-divisor.
(That is,  $(X', \Delta')$ is a log-Fano pair.) 
>From the technical point of view the difference between
$-K_X$ being  birationally pseudo-effective and  $X$ being of Fano type can be substantial, but in our examples none of the difficulties appear.
\end{rem}

\begin{say}[Outline of the proof of Theorems \ref{main.surf.thm} and \ref{main.thm}]
\rm Start with  $X\to \p^2$ over $\c$ and restrict to a generic line $\p^1\subset \p^2$.
We get a 2-dimensional conic bundle  $S\to \p^1$ over the function field $\c(s,t)$.
A slight complication is that $-K_X|_S=-K_S+F$ where $F$ is a general fiber,
but this is easy to deal with. Next we study when  $-K_S+F$ is   birationally pseudo-effective   for  2-dimensional conic bundles  $S\to \p^1$ over $\c(s,t)$, or, more generally, over any field $k$. 
We show in Corollary \ref{subc.lem.1} that this holds iff  $S$ contains a double section $C\subset S$ whose normalization has genus  $\leq 2$. The double section $C\subset S$ then extends to a generically finite 
double section $D\subset X$  whose normalization is a birational to a K3 surface. (The projection $\pi_D:D\to \p^2$ has degree 2 and the ramification curve is a sextic, but $\pi_D$  need not be finite.)
Then we study how the branch curve of $X\to \p^2$ and the 
branch curve of $D\to \p^2$ intersect.

In order to prove Theorem \ref{main.surf.thm}, we extend a 2-dimensional conic bundle over
$\p^1_{\q}$ to a 3-dimensional conic bundle over
$\p^1_{\z}$ and argue as above.  The answer seems to depend on subtle properties of $S\to \p^1_{\q}$. In particular, the branch divisor $B_S\subset \p^1_{\q}$ alone is not enough
to decide what happens; see Example \ref{22.exmp}.
\end{say}

\subsection*{Surface conic bundles}

Fix an arbitrary field $k$. Assume for simplicity that $\chr k\neq 2$. 
Let $S$ be a smooth surface over $k$ and 
$S\to \p^1$  a minimal conic bundle. The number of singular geometric fibers, which is the degree of the branch locus $B_S$, is denoted by $\delta(S)$.
  A typical fiber is denoted by $F$.
Thus $(K_S^2)=8-\delta(S)$ and $(K_S\cdot F)=-2$. 

\begin{lem} \label{lem5}
Let $k$ be a field and $\pi: S\to \p^1$ a minimal conic bundle.
Fix a natural number $m$ such that $\delta(S)> 12+6m$. 
Assume that $-K_S+mF$  is pseudo-effective.

Then there is a unique  irreducible curve $C\subset S$ such that
$(C^2)<0$. Furthermore, $\pi:C\to \p^1$ has degree 2   and $|-K_S+mF|=C+|bF|$ for some $b\geq 0$. 
In particular, 
$-K_S+mF$ is  effective.
\end{lem}

\begin{proof} By assumption there is a sequence of effective $\q$-divisors
$D_t$ converging to $-K_S+mF$. 
Note that $(-K_S+mF)^2=8-\delta(S)+4m<0$, hence  $(D_t^2)<0$ for some $t$.
So there is a $k$-irreducible component $C\subset \supp D_t$ with $C^2<0$. Thus the cone of curves is generated by $C$ and a fiber $F$. 

Write $C\simq a(-K_S+m'F)$. Here $a$ is an integer but $m'$ could be rational. 
The degree of the dualizing sheaf  of $C$ is
$$
\deg \omega_C=2p_a(C)-2=C(C+K_S)=a(a-1)(K_S^2)+2a(2a-1)m'.
$$
Since $C_{\bar k}$ has at most $2a$  irreducible components, $\deg \omega_C\geq -4a$, hence
$$
a(a-1)(K_S^2)+2a(2a-1)m'\geq -4a.
$$
If $a\geq 2$ then this rearranges to
$$
\delta(S)-8=-(K_S^2)\leq \bigl(4+\tfrac2{a-1}\bigr)m'+\tfrac4{a-1}\leq 6m'+4.
$$
Note that  $  -K_S+mF\simq \frac1{a}C +(m-m')F$, hence
$m'\leq m$ since $-K_S+mF$  is pseudo-effective.
Thus, if  $\delta(S)> 12+6m$ then $a=1$.


Note that we also could have used Bend-and-break
(as stated in \cite[II.5.5.3]{rc-book}) to show that $a\leq 2$.
\qed \end{proof}

\begin{exmp} \rm There are  minimal conic bundles $S\to \p^1$
with $\delta(S)= 12$ for which  $-2K_S$ is effective but
$-K_S$ is not effective. 

To construct such examples,
let $Q\subset \p^2$ be an irreducible  degree 4 point. Then $B_Q\p^2\to \p^1$
is a minimal conic  bundle with with $\delta=3$. 
The exceptional curves give a conjugate set of 4 sections, each with
self-intersection $-1$.

Pull it back by a general degree $r$ map
$\p^1\to \p^1$ to get  $S\to \p^1$. We get a minimal conic  bundle $S\to \p^1$ with $\delta(S)=3r$ and a conjugate set of 4 sections, each with
self-intersection $-r$. The sum of these 4 sections is  in $|-2K_S+(r-4)F|$.
Thus the bound $\delta(S)>12+6m$ in Lemma \ref{lem5}  is sharp.

(It is interesting to note that for $r=4$ contracting the 4 sections gives $S\to T$ where $T$ is a singular
Enriques surface. It has 4 singular points and $2K_T\sim 0$.) 
\end{exmp}

\begin{cor} \label{subc.lem.1} Let $k$ be a field and $S\to \p^1$ a minimal conic bundle. Assume that $\delta(S)>12+6m$. Then $-K_S+mF$ is birationally  pseudo-effective
iff $S$ contains a double section $C\subset S$ whose normalization has genus  $\leq m+1$.
\end{cor}

\begin{proof} Assume that $S\map S'$ is birational and $-K_{S'}+mF'$ is  pseudo-effective.
By \cite{MR0220734} (see also Theorem \ref{sark.thm}),
$S'$ is another  minimal conic bundle with the same $\delta$,
$S\to \p^1$ and $S'\to \p^1$ have the same generic fiber
 and 
 $-K_{S'}+m'F'$ is linearly equivalent to an irreducible double section $C'$ 
 for some $m'\leq m$ by Lemma \ref{lem5}.
Thus $C$ is obtained as the birational transform of $C'$.
Its   normalization has genus  $\leq m+1$   by Lemma \ref{subc.lem}.

Conversely, let $C\subset S$ be a double section. 
We can resolve the singularities of $C$ by   performing  elementary transformations at its singular points; see Paragraph \ref{el.transf.say}.
(Clearly, these points are on smooth fibers.) At the end we have
$C'\subset S'$ and $C'$ is smooth. If $p_a(C')\leq m+1$ then
$C'$ is a sub-curve of  $|-K_{S'}+mF'|$ by Lemma \ref{subc.lem}. \qed \end{proof}

\begin{lem} \label{subc.lem}
Let $C\subset S$ be a double section. Then
$C$ is a  sub-curve of  $|-K_{S}+mF|$ iff $p_a(C)\leq m+1$.
\end{lem}

\begin{proof} $C\sim -K_S+bF$ for some $b\in \z$. Thus
$2p_a(C)-2=C(C+K_S)=b(C\cdot F)=2b$ hence $b\leq m$ iff $p_a(C)\leq m+1$.\qed \end{proof}

\subsection*{Threefold conic bundles}

 We use  Corollary \ref{subc.lem.1}    to prove a similar result for 3-dimensional conic bundles.

\begin{thm} \label{main.surf.cor2} Let  $\pi:X\to \p^2$ be a 
 smooth, minimal conic bundle over a field of characteristic  $\neq 2,3,5$ 
whose  branch curve $B_X\subset \p^2$
has degree  $\geq 19$.  The following are equivalent.
\begin{itemize}
\item[\refstepcounter{equation} \label{main.surf.cor2.1} 1.]    $-K_X$ is  birationally pseudo-effective. 
\item[\refstepcounter{equation} \label{main.surf.cor2.2} 2.] There is a  generically finite  double section  $D\subset X$ with normalization $\tau:\bar D\to D$  such that the branch curve $B_D$ of $\pi\circ\tau:\bar D\to\p^2$ has degree $\leq 6$.
\item[\refstepcounter{equation} \label{main.surf.cor2.3} 3.]  $X$ is birational to a  smooth, minimal conic bundle $\pi':X'\to \p^2$
 (with the same branch curve) such that $|-K_{X'}|\neq \emptyset$. 
\end{itemize}
\end{thm}

\begin{proof} Assume that $\phi:X\map X'$ shows that  $-K_X$ is birationally pseudo-effective. By Complement \ref{sark.thm.c}, we may choose $X'$ such that 
 $\pi':X'\to \p^2$ is also a conic bundle with the same branch curve.
Write  $-K_{X'}\simq \Delta'$ where $\Delta'$ is pseudo-effective.

Let $L\subset \p^2$ be the generic line. By restriction
we get $\pi'_S:S'\to L$.
Since  $-K_{S'}+F=-K_{X'}|_{S'}$, we see that
$-K_{S'}+F\simq \Delta'|_{S'}$ is   pseudo-effective. 
Moreover, if $\Delta'_t$ is a sequence of $\q$-effective divisors converging to
$\Delta'$ then  $\Delta'_t|_{S'}$ is  a sequence of $\q$-effective divisors converging to $\Delta'|_{S'}$.

Thus,  by Lemma \ref{lem5}, 
$ \Delta'_t|_{S'}$ has an irreducible component   $C'\subset S'$
with negative self-intersection and its  normalization
has genus $\leq 2$  by Lemma \ref{subc.lem}.  Since $L$ is the generic line, $C'$ 
is the restriction of an irreducible component $D'$ of  $\Delta'_t$.
Thus $D'$ is a  
double section  and its birational transform on $X$ is a double section
$D\subset X$. Furthermore,
 we know that
the preimage of $L$ in the 
 normalization $\tau:\bar D\to D$  has genus $\leq 2$.
Thus $L$ intersects the 
 branch curve of $\pi\circ\tau:\bar D\to\p^2$ in  $\leq 6$ points. 
This shows that  (1) $\Rightarrow$ (2).

 Assume  next that  $\pi$ has a double section  $D\subset X$ as in (2).
By  Lemma \ref{d.sec.elem.lem}, there is  a sequence of elementary transformations
 $X\map X'$ such that the branch curve of $\pi': D'\to\p^2$ has degree $\leq 6$. By Lemma \ref{subc.gen.lem} this implies that 
$D'$ is a  sub-divisor of  $|-K_{X'}|$.  Thus (2) $\Rightarrow$ (3)
and (3) $\Rightarrow$ (1) holds by definition. \qed \end{proof}

\begin{lem} \label{d.sec.elem.lem}
 Let $\pi:X\to Z$ be a 
 conic bundle. Assume that
$X$ and $Z$ are smooth 
and $\dim Z= 2$. Let $D\subset X$ be a generically finite double section. 
Then there is a sequence of elementary transformations
$X\map X'$ such that the resulting $D'\subset X'$ is normal, except possibly along finitely many fibers.
\end{lem}

\begin{proof} Let $C\subset \sing D$ be a curve not contained in a fiber of $\pi$. Let $F'\subset F\subset X$ be an irreducible component of a fiber that is not contained in $D$ but meets $C$
at a point $p$.
Since $D$ is singular at $p$, the local intersection number $(D\cdot F')_p$  is at least 2.  Since $D$ is a double section, $(D\cdot F)\leq 2$. Thus
$F$ is irreducible and 
$D\cap F=\{p\}$, hence $\pi_C:C\to\pi(C)$ is birational.
As we discuss in Paragraph \ref{el.transf.say},
the elementary transformation centered at $C$ decreases the degree of the branch curve of $D\to \p^2$. After finitely many such steps 
we get $\pi':X'\to Z$ such that $D'\subset X'$ is normal, except possibly along finitely many fibers.\qed \end{proof}

\begin{lem} \label{subc.gen.lem}
Let  $\pi:X\to \p^2$ be a 
 minimal conic bundle and $D\subset X$ a  generically finite double section. Then
$D$ is a  sub-divisor of  $|-K_{X}|$ iff the branch curve $B_D\subset \p^2$ of $\pi_D: D\to\p^2$ has degree $\leq 6$.
\end{lem}

\begin{proof} Since $X$ is minimal, we know that $D\sim -K_X+b\pi^*H$ for some $b\in \z$ where $H$ is the  class of a line in $\p^2$. 
Thus $K_D\sim b\pi_D^*H$. 

Let $D\to D'\to \p^2$ denote the Stein factorization. Then 
$K_D\sim b\pi_D^*H$ pushes forward to $K_{D'}\sim b\pi_{D'}^*H$.
By the Hurwitz formula
$K_{D'}\sim \pi_{D'}^*\bigl(K_{\p^2}+\tfrac12 B_D\bigr)$.  Thus
$b\leq 0$ iff $\deg B_D\leq 6$. \qed \end{proof}

\subsection*{K3 surfaces on conic bundles}

In order to prove Theorem \ref{main.thm}, it remains to  show  that
(\ref{main.surf.cor2.2}) does not hold for suitable
branch curves $B\subset \p^2$. 
Example \ref{K3.exmp} shows that there are conic bundles with high degree smooth branch curve for which $|-K_X|$ does contain a K3 surface.
Thus we need to focus on more subtle properties of $B$.
We present 2 approaches. The first uses branch curves with many nodes;
we prove that  all the nodes have to lie on a sextic curve.
This leads to  quite explicit examples starting with $\deg B=12$; see Corollary  \ref{good.nodes.lem.cor}. 

Note also that, in a flat family of conic bundles, double sections are parametrized by countably many components of the Chow variety. Thus if we find
one conic bundle without  certain type of  double sections then the very general conic bundle also has no double sections of the same type.

The other approach proves that there has to be a 
sextic curve that is everywhere tangent to $B_X$.
This is easily seen to be a non-empty Zariski open condition for
$\deg B\geq 15$, though I could not write down explicit examples.
As pointed out by \cite[Sec.5]{MR651652}, the results of
\cite[Sec.3]{MR0321934} imply that every smooth plane curve of 
degree $\geq 3$ is the branch curve of a minimal conic bundle.

\begin{lem} \label{good.nodes.lem}
Let $\pi:X\to Z$ be a minimal, smooth conic bundle with branch curve
$B_X$. Let $D\subset X$ be a double section 
with normalization $\tau:\bar D\to D$ 
and $B_D\subset Z$ the branch curve of  $\pi|_D\circ \tau:\bar D\to Z$.
Then $\sing (B_X)\subset B_D$.
\end{lem}

\begin{proof} If $p\notin B_D$ then both of the \'etale local branches of
$\bar D\to Z$ give  \'etale local sections of $\pi$. 
As we discuss in Claim \ref{extend.say.5}, there are no \'etale local
 sections over the nodes of $B_X$ for $X$ smooth. \qed \end{proof}

\begin{cor} \label{good.nodes.lem.cor} Let $X\subset \p^2_{\mathbf x}\times  \p^2_{\mathbf y}$ be given by an equation of bidegree $(d,2)$
$$
\tsum_i  g_i(x_0, x_1, x_2) y_i^2=0.
$$
Assume that the curves 
$B_i:=\bigl(g_i(x_0, x_1, x_2)=0\bigr)\subset  \p^2_{\mathbf x}$
are smooth and they intersect each other transversally in $3d^2$ distinct points.

Let $D\subset X$ be a double section 
and $B_D\subset  \p^2_{\mathbf x}$ the branch curve of  $\pi|_D\circ \tau:\bar D\to Z$.
Then  $\deg B_D\geq 2d$.
\end{cor}

\begin{proof} The assumptions imply that $X$ is smooth. 
The branch curve of $X$ is $B_X=B_0+B_1+B_2$. 
By Lemma \ref{good.nodes.lem}, $B_D$ passes through all the nodes of $B$.
In particular, $B_D$ intersects  $B_0$ in the $2d^2$ points
$B_0\cap (B_1\cup B_2)$. Thus $\deg B_D\geq 2d$ unless  $B_0\subset B_D$. 
We can repeat the argument for $B_1$ to get that 
$\deg B_D\geq 2d$ unless  $B_0+B_1\subset B_D$. The letter case  also 
implies that $\deg B_D\geq 2d$. \qed \end{proof}

\medskip
For conic bundles with smooth branch curves we have a less precise condition.
(The 2 results are closely related 
since  passing through a node guarantees that the 
intersection multiplicity is $\geq 2$.)

\begin{prop} \label{main.3.prop} Let  $\pi:X\to \p^2$ be a 
 smooth, minimal conic bundle 
with   branch locus $B_X\subset \p^2$.  
Assume that there is no reduced curve $C\subset \p^2$  of degree $\leq 6$
such that all  points of $C\cap B_X$
have intersection multiplicity $\geq 2$.
Then $-K_X$ is not effective.
\end{prop}

\begin{proof}  Assume to the contrary that  $-K_X\sim D$ where $D$ is an effective  $\z$-divisor. Since $X$ is minimal, there are no rational sections. 
Thus there is a unique irreducible component $D_0\subset D$ that is a double section of $\pi$. 
Let $\tau:\bar D_0\to D_0$ denote the normalization and 
let $B_D\subset \p^2$ be the branch curve of 
 $\tau\circ \pi|_{D_0}:\bar D_0\to \p^2$. 
Note that  $D_0\sim -K_X-mH$ for some $m\geq 0$ and $\deg B_D\leq 6$ by
Lemma \ref{subc.gen.lem}.

We claim that  all  points of $B_D\cap B_X$
have intersection multiplicity $\geq 2$. To see this, assume to the contrary that there is a
point  $p\in B_D\cap B_X$ where $B_D, B_X$ are smooth and intersect transversally. 
In a neighborhood of $p$ the surface  $D_0$ is a smooth double cover
ramified along $B_D$. 
Then $(\pi|_{D_0})^{-1}(B_X)$ is a smooth double cover of $B_X$
ramified at $p$. On the other hand, $(\pi|_{D_0})^{-1}(B_X)=D_0\cap \pi^{-1}(B_X)$
hence its normalization factors through the \'etale cover
$\tilde B_X\to B_X$.  This is a contradiction.
(This argument in fact shows that  these intersection multiplicities are even.)
\qed \end{proof}

\medskip
Next we check that the assumptions of Proposition \ref{main.3.prop}
hold for a general branch curve of degree $\geq 15$.

\begin{lem}\label{den.curve.ok} Let $B\subset \p^2$ be a general curve of degree $d$. 
Then there is no reduced curve $C\subset \p^2$  of degree $<\tfrac12 d-1$
such that all  points of $C\cap B$
have intersection multiplicity $\geq 2$.
\end{lem}

\begin{proof}
Fix a reduced  curve $C$ and let $L$ be a line bundle of degree $m$. Let $W(L)\subset H^0(C, L)$ denote the subvariety consisting of sections without simple zeros.  Note that any $s\in W(L)$ has at most $\frac12 m$ zeros.
The map that sends a section of $L$ to its set of zeros has 1-dimensional fibers. Thus $\dim W(L)\leq \frac12 m+1$. 

Fix a reduced plane curve $C\subset \p^2$ of degree $c$. Let $W(C,d)$ denote the  set of degree $d$ curves $B$
such that all  points of $C\cap B$
have intersection multiplicity $\geq 2$. Applying the above estimate to
$L=\o_{\p^2}(d)|_C$ we get that
$\dim W(C,d)\leq \frac12 cd+\binom{d-c+2}{2}$.
Next let  $W(d)$ denote the  set of degree $d$ curves $B$
such that all  points of $C\cap B$
have intersection multiplicity $\geq 2$ for some curve $C$ of degree $c$.
Then
$$
\dim W(d)\leq \tfrac12 cd+\tbinom{d-c+2}{2}+\tbinom{c+2}{2}-1.
$$
After expanding the binomials we see that 
$$
\dim W(d)<\tbinom{d+2}{2}-1\qtq{for $d>2c+2$.}
$$
\qed \end{proof}
(Note that the bound  $\tfrac12 d-1$ is sharp for $d=4$ but
one should be able to prove slightly better bounds for larger values of $d$.)

\begin{exmp} \label{K3.exmp}
\rm  Let $Z$ be a smooth variety and $E$ a rank 3 vector bundle one $Z$.
Set $P:=\p_Z(E)$ with projection $\pi:P\to Z$ and 
note that $\omega_{P/Z}\cong \o_P(-3)\otimes \pi^*\det E$.
Let $L$ be a line bundle on $Z$ and $X\subset P$ the zero set of a section of $\o_P(2)\otimes \pi^*L$.
Then $\omega_X\cong \bigl(\o_P(-1)\otimes \pi^*(L\otimes \det E)\bigr)|_X$ and
$\pi_*\bigl(\omega_X^{-1}\bigr)\cong (\omega_Z\otimes L\otimes \det E)^{-1}\otimes E$.

Note that $H^0\bigl(P, \o_P(2)\otimes \pi^*L\bigr)=H^0\bigl(Z,  L\otimes \sym^2 E\bigr)$ and
$H^0\bigl(X, \o_X(-K_X)\bigr)=H^0\bigl(Z,  \bigl(L(K_Z)\otimes \det E\bigr)^{-1}\otimes E\bigr)$. 
We would like $X$ to be smooth, this suggests that 
$L\otimes \sym^2 E$ should be generated by global sections, 
hence $L$ should be positive. By contrast, the condition
$|-K_X|\neq \emptyset$ suggests that $L$ should be negative.
It seems that both of these can be satisfied  only if $E$ is rather unstable.

For example, take $Z=\p^2$, set $E=\o_{\p^2}(c)\oplus\o_{\p^2}(3)\oplus\o_{\p^2}$  where $c>3$ and $L=\o_{\p^2}$.
Then $L\otimes \sym^2 E$ is generated by global sections, so a general section gives a smooth conic bundle $X$. Furthermore, 
$$
\pi_*\bigl(\omega_X^{-1}\bigr)\cong \o_{\p^2}(3-3-c)\otimes E
\cong \o_{\p^2}\oplus\o_{\p^2}(3-c)\oplus\o_{\p^2}(-c)
$$
has a unique section. Thus $X\to \p^2$ is a conic bundle such that
$|-K_X|$ contains a unique K3 surface.
It is obtained by intersecting $X$ with   
the divisor in $P$ corresponding to the $\o_{\p^2}(3)\oplus\o_{\p^2}$ summand of $E$.
The branch curve of $X$ is given by the equation
$$
\det
\left(
\begin{array}{ccc}
g_{2c} & g_{c+3} & g_c\\
  g_{c+3} & g_6 &  g_3\\
 g_{c} & g_3 &  g_0
\end{array}
\right)\ =0,
$$
where  $g_i:=g_i(x,y,z)$ denotes a homogeneous polynomial of degree $i$.
(We can thus assume that $g_0=1$.) Eliminating $g_c, g_3$ in the last row
writes $g_0$ times the determinant, hence the equation of $B_X$, in the form 
$$
(g_6g_0-g_3^2)(g_{2c}g_0-g_c^2)-(g_{c+3}g_0-g_cg_3)^2=0.
$$
Thus the degree 6 curve  $g_6g_0-g_3^2=0$ intersects $B_X$ at the points where
$g_{c+3}g_0-g_cg_3=0$ and all intersection multiplicities are even;
as needed in Lemma \ref{main.3.prop}.
\end{exmp}

\subsection*{Examples of conic bundle surfaces}

Working backwards from Theorem \ref{main.thm} we see that there are
surface conic bundles  $S\to \p^1$  over $\c(s,t)$
such that $-K_S$ is not birationally $\q$-effective. 
We will exhibit similar examples over $\q$ and $\q_p$.

\begin{exmp} \label{22.exmp}
\rm Let $k$ be a field and $S\to \p^1$  a conic bundle over $k$ that becomes trivial after a quadratic extension  $k'\supset k$.
Then $S_{k'}$ has a section $C'$, thus $S$ has a conjugate pair of sections $C$.
The normalization of $C$ has $p_a=-1$.
Thus Corollary \ref{subc.lem.1}  implies that $S\to \p^1$ is birational to a conic bundle  $S'\to \p^1$
such that $-K_{S'}$ is effective. (Typically the base locus of
$|-K_S|$ consists of
the disjoint conjugate sections and the moving part of  2  fibers.)

There are such examples over $\q$, even with arbitrary branch locus. 
Let $g(s)\in \q[s]$ be a polynomial of degree $2d$ with simple roots only.
Choose a prime $p$ such that $\sqrt{p}$ is not contained in the splitting field of $g$. Let $S_g\to \p^1$ be the projective model of  the surface
$$
\bigl(g(s)z^2=x^2-py^2\bigr)\subset \p^1_{xyz}\times \a^1_s.
$$
Then $S_g\to \p^1$ is minimal and the singular fibers lie exactly over the roots of $g(s)$. 
 Thus there are 
 surface conic bundles  $S\to \p^1$  with  $\delta(S)$ arbitrarily large and 
branch locus $B_S$ in general position  for which $-K_S$ is 
birationally effective.

Note also that the above argument implies that 
 $-K_S$ is birationally effective 
for every conic bundle over $\r$ since every surface conic bundle over $\c$
has a section.
The situation over finite  fields is unclear to me.
\end{exmp}

The examples where $-K_S$ is not birationally $\q$-effective rest on the following observation.

\begin{lem}\label{gen.br.big.lem}
 Let $Z$ be a 2-dimensional regular scheme and
$\pi:X\to Z$ a conic bundle.  Let $C\subset Z$ be an irreducible 
 1-dimensional subscheme such that $X_C:=\pi^{-1}(C)\to C$ has no  rational sections.
 Let $W\subset C$ be the set of points 
$z\in C$ such that $X_z$ is a double line and $X$ is regular along $X_z$.
Let $D\subset X$ be a double section 
with normalization $\tau:\bar D\to D$ 
and $B_D\subset Z$ the branch locus of  $\pi|_D\circ \tau:\bar D\to Z$.

Then $C\not\subset B_D$ but $W\subset B_D$.
\end{lem}

\begin{proof} If $C\subset B_D$ then 
 $\red\bigl(D\cap X_C\bigr)$
is  a rational section of  $X_C\to C$. This is contrary to our assumptions and  $W\subset B_D$ follows from 
 Claim \ref{extend.say.3}. \qed \end{proof}

\medskip

\begin{exmp} \label{over.Q.exmps.exmp}
\rm We will apply Lemma \ref{gen.br.big.lem} with $Z=\p^1_{\z_p}$ and $C=\p^1_{\f_p}$
for some prime $p\geq 3$.
Fix  a natural number $d$ and set
$P:=\p_Z\bigl(\o_Z\oplus\o_Z(d)\oplus\o_Z(d+1)\bigr)$. The conic bundle
$X\subset P$ will be given by a section of $\o_P(2)$. For simplicity we choose a section of $\o_Z+\o_Z(2d)\oplus\o_Z(2d+2)\subset \pi_*\o_P(2)$. 
Choosing an affine coordinate  $s$ on $\p^1$, one can give such an $X$ by an equation
$$
X:=\bigl(z^2=a(s)x^2+b(s)y^2\bigr)\qtq{where} \deg a(s)=2d, \deg b(s)=2d+2.
$$
We choose $a(s)$ and $b(s)$ as follows.
\begin{itemize}
\item[\refstepcounter{equation} \label{over.Q.exmps.exmp.1} 1.] $\bar a(s)$ has only simple zeros where $\bar{\ }$ denotes reduction mod $p$,
\item[\refstepcounter{equation} \label{over.Q.exmps.exmp.2} 2.] $b(s)=(s^2-1)a(s)+pc(s)$ for some $c(s)$,
\item[\refstepcounter{equation} \label{over.Q.exmps.exmp.3} 3.] $s^2-1, \bar a(s), \bar c(s)$ are pairwise relatively prime and
\item[\refstepcounter{equation} \label{over.Q.exmps.exmp.4} 4.] $\bar a(1)$ is not a square in $\f_p$.
\end{itemize}

\medskip

\Claim \label{over.Q.exmps.exmp.5} 
Let $\pi:X_{\q_p}\to \p^1_{\q_p}$
be the generic fiber of  the above $X\to Z$ and 
$D_{\q_p}\subset X_{\q_p}$  a double section with
normalization $\tau:\bar D_{\q_p}\to D_{\q_p}$. Then
$g(\bar D_{\q_p})\geq d-1$.

\medskip

\begin{proof} The closure of $D_{\q_p}$ gives a double section $D\subset X$.
Let $B_D\subset Z$ be the branch locus of  $\pi|_D\circ \tau:\bar D\to Z$.
We aim to apply  Lemma \ref{gen.br.big.lem}.

Let $\alpha$ be a root of $ a(s)$. Near  $\alpha$ the equation of
$X$ has the form
$$
z^2=(s-\alpha)u_1x^2+\bigl((s-\alpha)u_2+pu_3\bigr)y^2,
$$
where the $u_i$ are units. Thus the fiber 
$X_{\alpha}$ is a double line and $X$ is regular along $X_{\alpha}$ by
(\ref{extend.say.5}).
Over the point $(s=1)\in  \p^1_{\f_p}$ the fiber is $(z^2=\bar a(1)x^2)\subset \p^2_{\f_p}$.
Since $\bar a(1)$ is not a square, its only $\f_p$-point is
$(0{:}1{:}0)$ where $\bar X$ is smooth. Hence   $\bar X \to \p^1_{\f_p}$ has no
sections. 
Thus Lemma \ref{gen.br.big.lem} implies that
$$
2g(\bar D_{\q_p})+2=\deg B_D\geq \deg a(s)=2d.
$$
\qed \end{proof}
\end{exmp}

The choices (\ref{over.Q.exmps.exmp.1}--\ref{over.Q.exmps.exmp.4}) can be satisfied for
$a(s), b(s)\in \q[s]$, hence we proved the following
more precise form of Theorem \ref{main.surf.thm}.

\begin{cor} For every $g$ there are conic bundles
$S\to \p^1$ defined over $\q$ with  $\delta(S)=4g+6$ and such that every double section of $S$ has geometric genus $\geq g$. 

For $g=2$ this gives  conic bundles
$S\to \p^1$ with $\delta(S)=14$ defined over $\q$ such that
$-K_S$ is not birationally pseudo-effective. \qed
\end{cor}

\subsection*{Birational maps  of conic bundles}

We summarize the results on birational maps  of conic bundles that we used.
As before, all residue characteristics are assumed to be $\neq 2$.

\begin{say}[Extending conic bundles]\label{extend.say}
\rm  We will need to understand the following 
\medskip

\problem 
\label{extend.say.1} 
Let $Z$ be a regular surface, $W\subset Z$ a finite subset and
$\pi^0:X^0\to Z^0:=Z\setminus W$ a conic bundle with branch locus
$B^0_X$. We would like to extend $\pi^0:X^0\to Z^0$ to a conic bundle
$\pi:X\to Z$ and control the singularities of $X$ in terms of $B_X$.
(This is also interesting if $\dim Z>2$ but the 2-dimensional case is simpler.)
\medskip

We may assume that $W=\{p\}$ is a single point and $Z$ is local. 
The push-forward  $E^0:=\pi^0_*\omega_{x^0/Z^0}^{-1}$ is a locally free sheaf of rank 3. Set $E:=j_*E^0$ where $j:Z^0\into Z$ is the natural injection.
Then $E$ is a reflexive sheaf but, since $Z$ is regular and 2-dimensional,
$E$ is locally free, hence free. Set  $P:=\p_Z(E)$ and let $X\subset P$ be the closure of $X^0\subset P^0$. Choose an isomorphism 
$P\cong \p^2_{Z}$, then
$X$ is given by an equation
\begin{equation}
\tsum_{ij} g_{ij}(z_1, z_2)x_ix_j=0,
\label{extend.say.2}
\end{equation}
where $g_{ij}(z_1, z_2)\in \o_Z$. The  scheme structure of the branch locus is defined by
\begin{equation}
\det \bigl(g_{ij}(z_1, z_2)\bigr)=0.
\label{extend.say.3}
\end{equation}
The worst case is when the central fiber $X_p$ equals $\p^2$; thus $\pi:X\to Z$ is not even equidimensional. If this happens then all the $g_{ij}$ vanish at $p$
so $\det \bigl(g_{ij}(z_1, z_2)\bigr)\in m_p^3$. Thus $B_X$ has a triple 
(or higher) point at $p$.
Otherwise the central fiber $X_p$ is conic. 

$X_p$ is smooth iff  $p\not\in B_X$.

Next assume that $X_p$ is a pair if lines. 
Possibly after a quadratic residue field extension, in suitable formal coordinates
we can diagonalize the equation of $X$ as
$$
x_0^2=x_1^2+f(z_1, z_2)x_2^2 \qtq{and} B_X=\bigl(f(z_1, z_2)=0\bigr)
$$
where $f(0,0)=0$. We see that $X$ is regular along $X_p$  iff
$f\not\in (z_1, z_2)^2$; that is, iff $B_X$ is regular at $p$.
Furthermore,   $(z_1, z_2)\mapsto (1{:}1{:}0; z_1,z_2)$ is a formal  section.

Finally consider the case when $X_p$ 
 is a double line. Then $\det \bigl(g_{ij}(z_1, z_2)\bigr)\in m_p^2$ and $B_X$ has a double (or higher) point
 at $p$.  If the branch curve  has a node then, after a quadratic residue field extension, in suitable formal coordinates
we can write $X$ as 
$$
X=\bigl(x_0^2=z_1x_1^2+z_2x_2^2\bigr).
$$
The fiber of the projection $\pi:X\to \a^2_{\mathbf z}$  is a 
double line over the origin, 
a pair of lines over the coordinate axes $(z_1z_2=0)$ and smooth otherwise.
In contrast with the previous nodal case,  the equation has no solutions in the quotient field of the completion of $\o_{z,Z}$; see  Claim \ref{extend.say.5}. 

We have thus proved the following.
\medskip

\Claim \label{extend.say.4}
 Let $Z$ be a regular surface whose  residue characteristics are  $\neq 2$, $B\subset Z$ a curve 
with only nodal singularities and $W\in Z$ a finite subset containing the nodes of $B$. Let
$\pi^0:X^0\to Z\setminus W$ be a conic bundle with (scheme theoretic) branch locus
$B\setminus W$.  
Then $X^0$ extends to a unique
 conic bundle
 $\pi:X\to Z$.

 Furthermore,  $X$ is regular iff the following holds:
for every $p\in W$ that is a node of $B$, the projection $\pi:X\to Z$ has no sections over the quotient field of the strict henselization ${\o}^{\rm sh}_{p,Z}$
of the  local ring
${\o}_{p,Z}$.
\qed
\medskip

We also used the following  well known result.
\medskip

\Claim \label{extend.say.5}
Let $\bigl(R,m=(s,t)\bigr)$ be a regular local ring of dimension 2.
Then $x_0^2=sx_1^2+tx_2^2$ has no nonzero solutions in the quotient field of $R$.
\medskip

\begin{proof} After clearing denominators, we may assume that
the $x_i$ are in $R$. Let $c$ be the largest such that the $x_i$ are in $m^c$.
We can thus write  $x_i=p_i(s, t)+r_i$ where $p_i$ is a 
homogeneous polynomial of degree $c$ and  $r_i\in m^{c+1}$.
Then $x_0^2\in m^{2c}$ but $sx_1^2+tx_2^2\in m^{2c+1}$. Thus
in fact  $p_0\equiv 0$  and so  $sp_1^2+tp_2^2\in m^{2c+2}$.
Thus $sp_1^2+tp_2^2\equiv 0$ hence  $p_1$ and $p_2$ are both identically 0. 
This is a contradiction.\qed 
\end{proof}
\end{say}

\begin{say}[Elementary transformations of conic bundles]\label{el.transf.say}
\rm Let $\pi:X\to Z$ be a 
 conic bundle over a regular surface and  $C\subset Z$ a regular curve such that
$\pi$ is smooth over  $C$
and $s:C\to X_C$ a section of
$X_C\to C$. We first blow up $s(C)\subset X$ and then contract the birational transform of $X_C$. We get another conic bundle  $\pi^{(s)}:X^{(s)}\to Z$.
The rational map  $X\map X^{(s)}$ is called the {\it elementary transformation} with center $s(C)$.

Next consider the case when $C\subset Z$ is a geometrically reduced curve, 
$\pi$ is smooth over the generic points of $C$ and 
  $s:C\map X_C$ is a rational section of
$X_C\to C$.  Then there is a finite subset $W\subset Z$  such that  $Z^0:=Z\setminus W$, $C^0:= C\setminus W$ and $X^0:=X\setminus \pi^{-1}(W)$ satisfy the previous assumptions. We can thus construct the
elementary transformation $\pi^{0,(s)}:X^{0,(s)}\to Z^0$ of $\pi^0:X^0\to Z^0$ with center
$s(C^0)$. Finally, whenever possible,  we use the method of
Paragraph \ref{extend.say} to extend $\pi^{0,(s)}:X^{0,(s)}\to Z^0$
to a conic bundle $\pi^{(s)}:X^{(s)}\to Z$ called the  
 {\it elementary transformation} with center $s(C)$.
\medskip

\Claim \label{el.transf.say.1} 
Let $\pi:X\to Z$ be a 
 conic bundle  over a  surface, $C\subset Z$ a geometrically reduced curve and 
  $s:C\map X_C$ a rational section of  $X_C\to C$. Assume that
$X$ and $Z$ are regular and  $\pi$ is smooth over the generic points of $C$. 

Then the elementary transformation with center $s(C)$ exists and it is a 
conic bundle $\pi^{(s)}:X^{(s)}\to Z$ such that $X^{(s)} $ is also regular
and has the same branch curve as $X\to Z$. \qed
\medskip

The following result of \cite{Sar81}, whose idea goes back to \cite{MR0220734},  describes  birational transformations of conic bundles over the same base. 

\medskip

\Claim \label{el.transf.say.2} 
Let $\pi_i:X_i\to Z$ be smooth  
 conic bundles  over a  smooth surface. 
Let $\phi: X_1\map X_2$ be a  birational equivalence over $Z$. That is, the following diagram commutes
$$
\begin{array}{ccc}
X_1 & \stackrel{\phi}{\map} & X_2\\
\pi_1 \downarrow\hphantom{\pi_1} && \hphantom{\pi_2}\downarrow\pi_2\\
Z & = & Z.\\
\end{array}
$$
Then $\phi$ is a composite of elementary transformations. \qed
\end{say}

The key result about birational maps of  conic bundles is the following.
The surface case is due to  \cite{MR0220734}. (The minimal model program for surfaces over any field is established in \cite{Mori82}, thus the arguments of \cite{MR0220734} extend to any field.)
The much harder
3-fold case is treated in
 \cite{Sar81}. See \cite{MR1311348, MR1798984} for more conceptual proofs.

\begin{thm} \label{sark.thm}
Let $\pi:X\to Z$ be a minimal 
 conic bundle over a field of characteristic  $\neq 2,3,5$
 such that $X,Z$ are smooth and
$B_X+4K_Z$ is effective.   Let $\pi':X'\to Z'$ be a Mori fiber space
and $\phi:X'\map X$ a birational map.  Then $\pi'$ is also a conic bundle,
and there is a birational map $\phi_Z:Z'\map Z$ such that the following diagram commutes
\begin{equation}
\begin{array}{ccc}
X' & \stackrel{\phi}{\map} & X\\
\pi' \downarrow\hphantom{\pi'} && \hphantom{\pi}\downarrow\pi\\
Z' & \stackrel{\phi_Z}{\map} & Z.\\
\end{array}
\label{sark.thm.1}
\end{equation}
\end{thm}

For our applications we need the following more precise version for which I could not find an explicit reference.

\begin{complem} \label{sark.thm.c}
Assume in addition that $B_X+K_Z$ is ample. Then the diagram
(\ref{sark.thm.1}) can be factored as
\begin{equation}
\begin{array}{ccccc}
X' & \stackrel{\rho}{\map} & X''& \stackrel{\tau}{\map} & X\\
\pi' \downarrow\hphantom{\pi'} && \hphantom{\pi''}\downarrow\pi''&& \hphantom{\pi}\downarrow\pi\\
Z' & \stackrel{\rho_Z}{\to} & Z &= & Z,\\
\end{array}
\label{sark.thm.c.1}
\end{equation}
where 
$\tau: X''\map X$ is a composite of elementary transformations,
$\rho:X'\map X''$ is a birational contraction and
 $-K_{X''}=\rho_*\bigl(-K_{X'}\bigr)$.
\end{complem}

\begin{proof} 
Let $\pi:X\to Z$ be a 
 conic bundle over a field of characteristic 0 such that $X,Z$ are smooth.
Pick a point $z\in Z$ and let $p:Z_1:=B_zZ\to Z$ denote the blow-up.
The pull-back  $X\times_ZZ_1$ is birational to  a minimal conic bundle $\pi_1:X_1\to Z_1$. 
We want to describe its branch locus  $B_{X_1}$.
It is clear that $B_{X_1}\supset p^{-1}_*B_X$, the only question is what happens with the exceptional curve $E_z$. The following possibilities are listed in \cite[2.4--2.5]{Sar81}.

If $z\not\in B_x$ then $X_1=X\times_ZZ_1$ hence  $B_{X_1}=p^{-1}_*B_X$.

If $z\in B_x$ is a smooth point then $E_z\subset B_{X_1}$ iff
the fiber $X_z$ irreducible. (Thus again $B_{X_1}=p^{-1}_*B_X$ if the base field is algebraically closed.)

If $z\in B_x$ is a singular point then $X_1=X\times_ZZ_1$ is singular.
Using the local equation $x_0^2=z_1x_1^2+z_2x_2^2$ we get
$x_0^2=z'_1x_1^2+z'_1z'_2x_2^2$. We can rewrite this as
$z'_1(x'_0)^2=x_1^2+z'_2x_2^2$ where $x'_0=x_0/z'_1$. Thus 
$B_{X'}=p^{-1}_*B_X+E_z$.
In all these cases we see that for any sequence of blow ups
$p_r:Z_r\to Z$ we have
\begin{equation}
K_{Z_r}+B_{X_r}=p_r^*(K_Z+B_X)+(\mbox{effective exceptional divisor}).
\label{sark.thm.c.2}
\end{equation}
Applying this to a common resolution  $Z'\leftarrow Z_r\to Z$
we conclude that if $K_Z+B_X$ is ample then  $(Z, B_X)$ is the unique log-canonical model of $(Z_r, B_{X_r})$. In particular, the rational map
$\phi_Z: Z'\map Z$ in (\ref{sark.thm.1}) is a morphism.
This establishes the bottom row of (\ref{sark.thm.c.1}).

In order to get the top row, let  $A'$ be an ample divisor on $Z'$ and choose $m$ such that  $|H'|:=|-K_{Z'}+m(p')^*A'|$ is a very ample linear system on $Z'$.
Its push-forward  $\phi_*|H'|$ is a mobile linear system consisting of rational double sections of $\pi$. Its base locus consists of some horizontal curves
(that is curves $C_i\subset X$ such that $C_i\to \pi(C_i)$ is birational) and some vertical curves
(that is  curves that are contracted by $\pi$). 
By a small variation of Claim \ref{el.transf.say.2},
a sequence of elementary transformations along the  horizontal  curves leads to a factorization
\begin{equation}
\phi: X'  \stackrel{\rho}{\map}  X'' \stackrel{\tau}{\map}  X
\label{sark.thm.c.3}
\end{equation}
where $\tau$ is a composite of elementary transformations and
the  base locus  $\bs\bigl(\rho_*|H'|\bigr)$ is vertical. 
Thus  $\rho^{-1}$ gives an injection $X''\setminus \bs\bigl(\rho_*|H'|\bigr)\into X'$, hence
it has no exceptional divisors. Thus $\rho$ is a birational contraction and therefore  $-K_{X''}=\rho_*\bigl(-K_{X'}\bigr)$. \qed \end{proof}


\subsection*{Acknowledgements} I thank A.~Corti, Y.~Gongyo, A.~Skorobogatov, C.~Xu and the referees   for helpful comments, corrections  and references.


\bibliographystyle{amsalpha}

\def\cprime{$'$} \def\cprime{$'$} \def\cprime{$'$} \def\cprime{$'$}
  \def\cprime{$'$} \def\cprime{$'$} \def\cprime{$'$} \def\dbar{\leavevmode\hbox
  to 0pt{\hskip.2ex \accent"16\hss}d} \def\cprime{$'$} \def\cprime{$'$}
  \def\polhk#1{\setbox0=\hbox{#1}{\ooalign{\hidewidth
  \lower1.5ex\hbox{`}\hidewidth\crcr\unhbox0}}} \def\cprime{$'$}
  \def\cprime{$'$} \def\cprime{$'$} \def\cprime{$'$}
  \def\polhk#1{\setbox0=\hbox{#1}{\ooalign{\hidewidth
  \lower1.5ex\hbox{`}\hidewidth\crcr\unhbox0}}} \def\cdprime{$''$}
  \def\cprime{$'$} \def\cprime{$'$} \def\cprime{$'$} \def\cprime{$'$}
\providecommand{\bysame}{\leavevmode\hbox to3em{\hrulefill}\thinspace}
\providecommand{\MR}{\relax\ifhmode\unskip\space\fi MR }
\providecommand{\MRhref}[2]{%
  \href{http://www.ams.org/mathscinet-getitem?mr=#1}{#2}
}
\providecommand{\href}[2]{#2}

\end{document}